\documentclass[11pt]{amsart}

\usepackage{amssymb,amsmath,amscd,graphicx,
latexsym,amsthm,yhmath}
\usepackage{amssymb,latexsym,amsmath,amscd,graphicx}
\usepackage[mathscr]{eucal}
  \usepackage[all]{xy}
  \usepackage{paralist}

\def\OO{{\mathcal O}}

\def\F{\mathcal{F}}

\def\G{\mathcal{G}}

\def\I{\mathcal{I}}

\def\cP{\mathcal{P}}
\def\Pic0{\mathrm{Pic}^0}

\def\l{{\underline l}}

\def\(id){\,\mathrm{id}}

\theoremstyle{plain}

\newtheorem{theorem}{Theorem}[section]
\newtheorem{theoremalpha}{Theorem}

\newtheorem{proposition/example}[theorem]{Proposition/Example}
\newtheorem{proposition}[theorem]{Proposition}

\theoremstyle{definition}
\newtheorem{definition}[theorem]{Definition}

\newtheorem{remark}[theorem]{Remark}

\newtheorem{conjecture/question}[theorem]{Conjecture/Question}

\newtheorem{remark/definition}[theorem]{Remark/Definition}
\newtheorem{notation/assumptions}[theorem]{Assumptions/Notation}

\numberwithin{equation}{section}

 \theoremstyle{remark}

\begin{document}
\title{Singularities of divisors 
 of low degree on simple abelian varieties}
\author{Giuseppe Pareschi}
  \address{ Universit\`a di Roma Tor Vergata, V.le della Ricerca Scientifica, I-00133 Roma, Italy}
 \email{{\tt pareschi@mat.uniroma2.it}}
 \thanks{Partially supported by the MIUR Excellence Department Project awarded to the Department of Mathematics, University of Rome Tor Vergata, CUP E83C18000100006.}
 \begin{abstract} It is known by results of  Koll\'ar,  Ein, Lazarsfeld, Hacon and Debarre  that effective divisors representing principal and other low degree polarizations  on complex abelian varieties have mild singularities. In this note, we extend these results to all polarizations of degree $<g$ on simple $g$-dimensional abelian varieties,  
 settling  a conjecture of Debarre and Hacon.
\end{abstract}
\maketitle

%\tableofcontents
\setlength{\parskip}{.1 in}

It is known by results of  Koll\'ar,  Ein, Lazarsfeld, Hacon and Debarre  that effective divisors representing principal and other low degree polarizations  on complex abelian varieties have mild singularities (\cite[Theorem 17.13]{kollar}, \cite{el}, \cite{hacon1}, \cite{hacon2},\cite{dh}).
In this note, we prove another result in the same direction, conjectured by Debarre and Hacon in \cite[\S6]{dh}  and proved by them  for $\chi(\l)<2\sqrt g -1$ and also for low values of $g$.

\begin{theoremalpha}\label{main} Let $(A,\l)$ be a complex $g$-dimensional simple polarized abelian variety  with $\chi(\l)<g$. Then 
\begin{compactenum}
\item every effective divisor $E$ representing $\l$ is prime \emph{(Debarre-Hacon, \cite[Proposition 2]{dh})} and normal with rational singularities.
\item Let $m\ge 2$ and let $D$ be an effective divisor representing $m\l$. Then, unless $D=mE$ with $E$ representing $\l$,  one has $\lfloor \frac{1}{m}D\rfloor=0$   \emph{(\cite[Corollary 2]{dh})} and the pair $(A,\frac{1}{m}D)$ is log terminal.\end{compactenum}
\end{theoremalpha}

We refer to the previously quoted works, especially the introductions of \cite{el} and \cite{dh}, and to Section 10.1.B of the book \cite{laz2} for history, motivation, and applications.

The proof  makes use of all the  ingredients of the previously quoted papers, in particular (generic) vanishing theorems
 involving adjoint and multiplier ideals and  the linearity theorem for their cohomological support loci. 
  In this way, Theorem \ref{main} is a standard consequence of
   Theorem \ref{basic} below, which is the main content of the present note.  
   
   Let us recall that,  given a coherent sheaf $\F$ on an abelian variety $A$,  its cohomological support loci are the following subvarieties 
\[V^i(A,\F)=\{\alpha\in\widehat A\>|\>H^i(A,\F\otimes P_\alpha)\ne 0\}\]
of $\widehat A:=\Pic0 A$, where  $P_\alpha$ denotes the line bundle parametrized by $\alpha\in\widehat A$ via the choice of a Poincar\'e line bundle. We  also set 
\[V_{>0}(A, \F)=\bigcup_{i>0}V^i(A, \F).\]
Finally, we recall that a subvariety $X$ of an abelian variety $A$ is said to be \emph{geometrically non-degenerate} if, for all abelian subvarieties $K$ of $A$, $\dim (X+K)=\mathrm{min}\{\dim A, \> \dim X+\dim K\}$ \ (\cite[Lemma II.12]{ran} , \cite[(1.11)]{debarre}). 

  \begin{theoremalpha}\label{basic} Let $(A,\l)$ be a polarized  $g$-dimensional  abelian variety and let $L$ be a line bundle representing $\l$.   Let $Z$ be a non-trivial subscheme of $A$ with geometrically non-degenerate support. Assume also that $Z$ is not a divisor representing $\l$. 
\begin{enumerate}
\item  If $V_{>0}(A, \I_Z\otimes L)$ is empty then $\chi(\l)\ge g+1$.
\item If $V_{>0}(A, \I_Z\otimes L)$ is $0$-dimensional then $\chi(\l)\ge g$. 
\end{enumerate} 
\end{theoremalpha} 
 
The proof of Theorem \ref{basic}  is based on the study of the Fourier-Mukai-Poincar\'e transform of the derived dual of the sheaf $\I_Z\otimes L$. This operation produces (see \S1 below)  a coherent sheaf  in cohomological degree $g$ on the dual abelian variety $\widehat A$,   whose generic rank is $\chi(\I_Z\otimes L)$. Such a sheaf is usually denoted by $\widehat{(\I_Z\otimes L)^\vee}$. From a body of results on the Fourier-Mukai-Poincar\'e transform,  it follows that    in case (1), $\widehat{(\I_Z\otimes L)^\vee}$ is an ample vector bundle, while in case (2), it is   a $k$-syzygy sheaf, with $k$ sufficiently high. Applying to $\widehat{(\I_Z\otimes L)^\vee}$ the Le Potier vanishing theorem in the former case and  the Evans-Griffith syzygy theorem in the latter case,\footnote{Interestingly, the theorems of Le Potier and Evans-Griffith are related. In fact, via linear complexes as in \cite[Remark 4.2]{bgg}, one can show that the application of the Evans-Griffith theorem needed here is in turn implied by Le Potier vanishing via an argument of Ein (\cite[Example 7.3.10]{laz2}). } we obtain a lower bound for the generic rank of the sheaf $\widehat{(\I_Z\otimes L)^\vee}$, hence for  $\chi(\I_Z\otimes L)$. This, combined with an upper bound for  $\chi(\I_Z\otimes L)$ due to Debarre-Hacon (inequality (\ref{ineq1}) below), proves Theorem \ref{basic}.  

\noindent\textbf{Acknowledgments. } The author thanks  Giulio Codogni, Zhi Jiang, Mihnea Popa, and Stefan Schreieder for their useful and interesting comments.

\section{Background on the Fourier-Mukai-Poincar\'e transform} 
In this section we briefly recall the necessary background about some sheaf-theoretic properties of the Fourier-Mukai-Poincar\'e transform on abelian varieties. We refer to the papers quoted below or to the survey \cite{msri} for more details. 

A Poincar\'e line bundle $\cP$ on $A\times \widehat A$ defines a Fourier-Mukai functor 
\[\Phi^{A\rightarrow\widehat A}_{\cP}:D(A)\rightarrow D(\widehat A)\]
 which is an equivalence of categories (Mukai \cite{mukai1}). Its inverse is \ 
 $\Phi_{\cP^\vee}^{\widehat A\rightarrow A}[g]$, \ which can be expressed as 
\begin{equation}\label{involution}\Phi_{\cP^\vee}^{\widehat A\rightarrow A}[g]=(-1)^*_A\circ \Phi_{\cP}^{\widehat A\rightarrow A}[g] \>,
\end{equation} 
where $(-1)_A $ denotes the natural involution on $A$. This follows from the fact that $\cP^\vee=(-1,1)^*_{A\times \widehat A}\cP=(1,-1)^*_{A\times \widehat A}\cP$ (\cite[Lemma 14.1.2]{bl}).

\begin{definition}\label{gv-etc}(\cite[Definition 3.1]{pp2})
Let $\F$ be a coherent sheaf of $A$. 
The \emph{gv-index} of $\F$ is the integer 
\[gv(\F)=\mathrm{min}_{i>0}\,(\mathrm{codim}_{\widehat A}V^i(A,\F)-i)\]
Moreover: \\
- if $gv(\F)\ge 0$  then $\F$ is said to be  a  \emph{Generic Vanishing sheaf}, or simply  \emph{GV};\\
-  if $gv(\F)\ge 1$  then $\F$ is said to be \emph{Mukai-regular}, or simply \emph{M-regular}; \\
- if $V_{>0}=\emptyset$  then $\F$ is said to \emph{verify the Index Theorem with index 0},  or simply \emph{IT(0)}.
\end{definition}

\vskip0.2truecm We set $\F^\vee:=R\mathcal Hom_A(\F,\OO_A)$. We have the following duality result. 

\begin{theorem}\label{wit} \emph{(Hacon, Pareschi-Popa, see \cite[Theorem 2.2]{pp2}, \cite[Theorem A]{pp})} Let $\F$ be a coherent sheaf on an abelian variety $A$. Then $\F$ is GV if and only if $\Phi^{A\rightarrow \widehat A}_{\cP^\vee}(\F^\vee)$ is a sheaf in degree $g$, i.e.
\[\Phi^{A\rightarrow \widehat A}_{\cP^\vee}(\F^\vee)=R^g\Phi^{A\rightarrow \widehat A}_{\cP^\vee}(\F^\vee)[-g].\]
\emph{If this is the case, following Mukai, we use the following notation
\[ \widehat{\F^\vee}:=R^g\Phi^{A\rightarrow \widehat A}_{\cP^\vee}(\F^\vee).\>\footnote{In some of the quoted references,  $\widehat{\F^\vee}$ denotes instead the sheaf $R^g\Phi^{A\rightarrow \widehat A}_{\cP}(\F^\vee)$.
 However this will not cause  trouble since, as in (\ref{involution}), $\Phi_{\cP^\vee}^{A\rightarrow \widehat A}[g]=(-1)^*_{\widehat A}\circ \Phi_{\cP}^{A\rightarrow \widehat A}$, hence the two transforms produce essentially the same results.} \]}
 \end{theorem}
\begin{remark}\label{remark} (a) Assume that $\F$ is GV. By base change and Serre duality the support of the sheaf $\widehat{\F^\vee}$ is  $V^0(A, \F)$. Therefore the subvariety  $V^0(A, \F)$ is non-empty as soon as $\F$ is non-zero, since otherwise $\Phi^{A\rightarrow \widehat A}_{\cP^\vee}(\F^\vee)$ would be zero, hence $\F$ itself would be zero.\\
(b)  For a GV sheaf  $\F$, the Euler characteristic   $\chi(\F)$ is equal to the generic value of $h^0(A,\F\otimes P_\alpha)$, for $\alpha\in\widehat A$. Therefore
$\chi(\F)\ge 0$
 and  $\chi(\F)$ coincides  with the generic rank of $\widehat{\F^\vee}$. 
 \end{remark}

For a GV sheaf $\F$, the dictionary between the gv-index and the sheaf-theoretic properties of the transform $\widehat{\F^\vee}$ is summarized in the following statement.
\begin{theorem}\label{syzygy}  Let $\F$ be a coherent sheaf on an abelian variety $A$. \\
(1) \emph{(Pareschi-Popa, \cite[Corollary 3.2]{pp2})}  For $k\ge 0$,   \ $gv(\F)\ge k$ if and only if $\widehat{\F^\vee}$ is a k-syzygy sheaf.\footnote{We refer to \cite[ Chapter 2, \S1.1]{oss} or the Appendix in \cite{pp2} for 
 $k$-syzygy sheaves. See also \cite{eg2}). }\\
 (2) $\F$ is M-regular if and only if  $\widehat{\F^\vee}$ is torsion-free \emph{(hence, in particular,  $\chi(\F)>0$).}\\
 (3)  $\F$ is IT(0) if and only if $\widehat{\F^\vee}$ is locally free; equivalently, if and only if $\Phi^{A\rightarrow \widehat A}_{\cP}(\F)=R^0\Phi^{A\rightarrow \widehat A}_{\cP}(\F)$.
\end{theorem}
Note that (2) is a particular case of (1) since a 1-syzygy sheaf is simply a torsion-free sheaf.  Item (3) is elementary: it follows immediately from Grauert's theorem on cohomology and base-change.

\bigskip A key ingredient of the proof of Theorems \ref{wit} and \ref{syzygy} is the  identification 
\begin{equation}\label{duality}
R^i\Phi^{A\rightarrow \widehat A}_{\cP}(\F)=\mathcal Ext^i_{\widehat A}(\widehat{\F^\vee},\OO_{\widehat A}),
\end{equation}
consequence of Grothendieck duality, (see \cite[Proposition 1.6(b)]{msri}   or \cite[Corollary 3.2]{pp2}).
\begin{remark}\label{ext} From (\ref{duality}) and  base-change, it follows that the support of the sheaf $\mathcal Ext^i_{\widehat A}(\widehat{\F^\vee},\OO_{\widehat A})$ is contained in $V^i(A,\F)$. 
\end{remark}

\begin{remark}\label{exchange}Clearly the roles of $A$ and $\widehat A$ can be exchanged, and all of the above could have been said for a sheaf $\G$ on $\widehat A$ as well, starting from  the Fourier-Mukai equivalence $\Phi_{\cP}^{\widehat A\rightarrow A}:D(\widehat A)\rightarrow D(A)$. For example,  
the cohomological support loci of the sheaf $\G$ are
\[V^i(\widehat A, \widehat \G)=\{a\in A\>|\> H^i(\widehat A,\G\otimes P_a)\ne 0\}.\]
\end{remark}

 \vskip0.2truecm Finally we recall that a sheaf $\F$ on an abelian variety $A$ is \emph{homogeneous} if it has a filtration \[0=\F_0 \subset \F_1 \subset\cdots\subset \F_n =\F\]
such that $\F_i/\F_{i-1}$ is isomorphic to a line bundle in $\widehat A$  for all $i \in\{ 1, \dots, n\}$. The following proposition is certainly well known, but we could not find a reference.

\begin{proposition}\label{homogeneous} A sheaf $\F$ on $A$ is a homogeneous vector bundle if and only if $\chi(\F)=0$ and $\dim V_{>0}(A,\F)\le 0$.
\end{proposition}
\begin{proof} The direct implication is obvious. Conversely, if $V_{>0}(A,\F)$ is 0-dimensional or empty then $\F$ is GV. Thus, by Remark \ref{remark} the condition $\chi(\F)=0$ means that the locus $V^0(A,\F)$ is a proper subvariety of $\widehat{A}$ (non-empty if $\F$ is non-zero). It is known that the GV condition has the following pleasant consequence: any component $W$ of $V^0(A,\F)$  of codimension $j$  is 
also a component of $V^j(A,\F)$ (see \cite[Proposition 3.15]{pp}  or \cite[Lemma 1.8]{msri}). Therefore, since $V_{>0}(A,\F)$ has dimension 0, $W$ must be $0$-dimensional (in fact an isolated point in $V^g(A,\F)$). 
Therefore, $\dim V^0(A,\F)=0$ and $\widehat{\F^\vee}$ is a $\OO_{\widehat A}$-module of finite length (it is supported at $V^0(A,\F)$). By a result of Mukai (\cite[Theorems 4.17\,,\,4.19]{mukai2}), this means that $\F^\vee$ is a homogeneous vector bundle. Equivalently,
$\F$ is a homogeneous vector bundle.
\end{proof}
\section{Proof of Theorems \ref{main} and \ref{basic}}

\noindent\textbf{Proof of Theorem \ref{basic}. }  From the exact sequence 
\begin{equation}\label{obvious}0\rightarrow \I_Z\otimes L\rightarrow L\rightarrow L_{|Z}\rightarrow 0,
\end{equation}
  it follows that for $i\ge 1$
 \begin{equation}\label{trivial}V^{i}(A,L_{|Z})=V^{i+1}(A,\I_Z\otimes L).
 \end{equation} 
  Therefore the hypotheses on $V_{>0}(\I_Z\otimes L)$ imply that in both cases (1) and (2),  the sheaf $L_{|Z}$ is M-regular, hence $\chi(L_{|Z})>0$ (see Theorem \ref{syzygy}(2)). Hence $\chi(\I_Z\otimes L)<\chi(L)$. On the other hand, $\chi(\I_Z\otimes L)>0$ since otherwise, by Proposition \ref{homogeneous}, the sheaf $\I_Z\otimes L$ would be a homogeneous vector bundle, and it is easy to verify that this happens if and only if $Z$ is a divisor representing $\l$. In conclusion, we have
  \begin{equation}\label{ineq2}
  0<\chi(\I_Z\otimes L)<\chi (L)\>.
  \end{equation}
 Note that  the generic values  of $h^0(\I_Z\otimes L\otimes P_\alpha)$ and $h^0(L_{|Z}\otimes P_\alpha)$, for $\alpha\in\Pic0 A$,  coincide with $\chi(\I_Z\otimes L)$ and $\chi(L_{|Z})$. By a result of 
 Debarre-Hacon, \cite[Lemma 5(e)]{dh},\footnote{This Lemma is stated  under the assumption that the abelian variety is simple, but in fact what is needed is that the support of $Z$ is geometrically non-degenerate.} the inequalities
 (\ref{ineq2}) imply that
\begin{equation}\label{ineq1}\chi(\I_Z\otimes L)\le \chi(L)-1-\dim Z \>.
\end{equation}

\noindent\emph{Proof of (1).  } The hypothesis means that $\I_Z\otimes L$ is IT(0) (Definition \ref{gv-etc}). By Theorem \ref{syzygy}(3), one has
\[\Phi^{A\rightarrow \widehat A}_{\cP}(\I_Z\otimes L)=R^0\Phi^{A\rightarrow \widehat A}_{\cP}(\I_Z\otimes L):=\G\]
 and $\G$ is  a locally free sheaf on $\widehat A$ of rank equal to $\chi(\I_Z\otimes L)$. Therefore, taking the inverse functor, 
 \[\Phi^{\widehat A\rightarrow A}_{\cP^\vee}(\G)=R^g\Phi^{\widehat A\rightarrow A}_{\cP^\vee}(\G)[-g]=\I_Z\otimes L[-g]\>.\]
  This means that the dual vector bundle $\G^\vee$ is a GV sheaf on $\widehat A$ (by Theorem \ref{wit} applied to the equivalence $\Phi_{\cP}^{\widehat A\rightarrow A}$, see also Remark \ref{exchange}). More, since the sheaf $\widehat \G=\I_Z\otimes L$ is torsion-free,  $\G^\vee$ is  M-regular  (Theorem \ref{syzygy}(2)). But a result of Debarre (\cite[Corollary 3.2]{debarre2})  says that a M-regular sheaf is ample. Therefore the vector bundle $\G^\vee $ is ample. Thus, by Le Potier's vanishing, its cohomological support loci $V^i(\widehat A, \G^\vee)$ are empty for $i>\mathrm{rk}\,\G^\vee-1=\chi(\I_Z\otimes L)-1$. 
  
On the other hand, by (\ref{duality}) and Remark \ref{ext}, the cohomological support locus $V^i(\widehat A, \G^\vee)$ contains the support of $\mathcal Ext^i_{\widehat A}(\I_Z,\OO_{\widehat A})=\mathcal Ext^{i+1}_{\widehat A}(\OO_Z,\OO_{\widehat A})$. As soon as $Z$ has an $(i+1)$-codimensional component such a sheaf is non-zero, hence $V^i(\widehat A, \G^\vee)$ is non-empty. Therefore
\[g-(\dim Z+1)\le \chi(\I_Z\otimes L)-1.\]
Together with (\ref{ineq1}), this proves (1).

\noindent \emph{Proof of (2). } For a coherent sheaf $\F$ on $A$, we set
 \[i_{\max}(\F)=\mathrm{\max}\{i\>|\> V^i(A,\F)\ne \emptyset\}.\]
 Assume  first that $i_{\max}(\I_Z\otimes L)=1$. By hypothesis, this implies that its gv-index  (Definition \ref{gv-etc}) is
 \[gv(\I_Z\otimes L)=g-1.\] Therefore, by Theorem \ref{syzygy}(1), $\widehat{(\I_Z\otimes L)^\vee}$ is a non-locally free \ ($g-1$)-syzygy sheaf. Therefore, by the Evans-Griffith syzygy theorem (\cite[Corollary 1.7]{eg1}, see also \cite[Appendix]{pp2})  its generic rank is $\ge g-1$,  i.e. 
 \[\chi(\I_Z\otimes L)\ge g-1.\]
 As we know that $\chi(L_{|Z})\ge 1$,  (2) follows in this case.

Assume otherwise that $i_{\max}(\I_Z\otimes L)>1$. As above, we have
 \[gv(\I_Z\otimes L)=g-i_{\max}(\I_Z\otimes L).\]  Again by Theorem \ref{syzygy}(1) and Evans-Griffith, we get 
\[\chi(\I_Z\otimes L)\ge g-i_{\max}(\I_Z\otimes L).\] On the other hand, by (\ref{trivial}), we have  $i_{\max}(\I_Z\otimes L)=i_{\max}(L_{|Z})+1\le \dim Z+1$. Therefore
\[\chi(\I_Z\otimes L)\ge g-\dim Z-1.\]
Together with (\ref{ineq1}), this proves (2) and concludes the proof of Theorem \ref{basic}.

\begin{remark} (a) In Theorem \ref{basic}, the hypothesis that $Z$ is not a divisor representing the polarization $\l$ is clearly necessary. \\
(b) The inequalities of Theorem \ref{basic} are  sharp in both cases (1) and (2), as it is shown by taking for $Z$  a  point $p\in A$.  Indeed a general polarized abelian variety  of type $(1,\dots , 1,g+1)$ is base point free
(\cite[Proposition 2]{dhs}). Since the line bundles $L\otimes P_\alpha$ are the translates of $L$, this is easily seen to be equivalent to the fact that $V_{>0}(A, \I_p\otimes L)$ is empty. Similarly,  a general polarized abelian variety of type $(1,\dots , 1,g)$ has  a 0-dimensional base locus  
(\cite[Remark 3(a)]{dhs}). As above, this means that $V_{>0}(A,\I_p\otimes L)$ is 0-dimensional. 
\end{remark}

\noindent\textbf{Proof of Theorem \ref{main}. }
The fact that Lemma \ref{basic} implies  Theorem \ref{main} is  known. We review this for the sake of self-containedness, referring to \cite[\S 10.1.B]{laz2} and \cite{dh} for more details. Indeed the last assertion of (1) (respectively the last assertion of (2)) of Theorem \ref{main} is equivalent  to the triviality of the adjoint ideal of $E$ (respectively  of the multiplier ideal of the $\mathbb Q$-divisor $\frac{1}{m}D$). We claim that both ideals satisfy all hypotheses of Lemma \ref{basic}  and therefore Lemma \ref{basic} implies Theorem \ref{main}. To prove the claim, let us denote $\I$ both ideals. 

To begin with, $\I$ cannot be $\OO_A(-E)$, where $E$ is a divisor representing $L$. This is obvious for the adjoint ideal, and  it holds for the multiplier ideal of $\frac{1}{m}D$ (if $D$ is not equal to $mE$, with $E$ representing $\l$)  because $\lfloor \frac{1}{m}D\rfloor=0$ (\cite[Corollary 3]{dh}). 

Moreover, by definition, any subvariety of a simple abelian variety is geometrically non-degenerate.

 It remains only to prove  that the locus $V_{>0}(A,\I\otimes L)$ is either empty or $0$-dimensional. This follows from the hypothesis that the abelian variety $A$ is simple and  the fact that $\I$ satisfies the following property:\\
(*)  \emph{ the locus $V_{>0}(A, \I\otimes L)$ is either empty or a proper linear subvariety, i.e. a finite union of translates of proper abelian subvarieties of $\Pic0 A$. }\\
To prove (*), we recall that  the generic vanishing and linearity theorems of Green and Lazarsfeld (\cite{gl2},\cite[Remark 1.6]{el}), combined with the Grauert-Riemenschneider vanishing theorem,  say that for 
 a smooth projective variety $X$, equipped with a generically finite morphism $f:X\rightarrow A$,  the locus $V_{>0}(A,f_*\omega_X)$ is either empty or a proper linear subvariety, i.e. a finite union of translates of proper abelian subvarieties of $\Pic0 A$. Hence
for the adjoint ideal, the property (*) follows from  via the  exact sequence
\[0\rightarrow \OO_A\rightarrow \I_Z\otimes L\rightarrow f_*\omega_{E^\prime}\rightarrow 0\]
where $f:E^\prime\rightarrow E$ is any resolution of singularities of $E$ (\cite[Proposition 9.3.48]{laz2}). If instead $\I$ is the multiplier ideal of the $\mathbb Q$-divisor $\frac{1}{m}D$,  property (*) follows 
 from the same theorems of Green and Lazarsfeld via the fact that such a multiplier ideal (twisted by the canonical bundle of $A$, which is trivial) is a direct summand of the pushforward of the canonical bundle of a smooth variety via a generically finite morphism.  This in turn goes back to the work of Esnault-Viehweg (\cite [(3.13)]{ev}). A more explicit reference is  \cite[page 226]{dh}.\footnote{See also
\cite[Theorem 1.3]{budur} for a more general linearity theorem along these lines.}  This proves (*).

\providecommand{\bysame}{\leavevmode\hbox
to3em{\hrulefill}\thinspace}

\end{document}